\documentclass[11pt]{amsart} 
\usepackage{amsmath,amsthm,amssymb,rawfonts}
\usepackage{pb-diagram}

\newcommand{\impl}{\rightarrow}

\newtheorem{theorem}{Theorem}[section]

\newtheorem{definition}[theorem]{Definition}
\newtheorem{proposition}[theorem]{Proposition}
\newtheorem{$^{*}$proposition}[theorem]{$^{*}$Proposition}

\newtheorem{$^{*}$corollary}[theorem]{$^{*}$Corollary}

\newtheorem{remark}[theorem]{Remark}

\numberwithin{equation}{theorem}
\begin{document}
\title{ VARIABLE-BASIS FUZZY FILTERS}
\author{Joaqu\'\i n Luna-Torres$^{\lowercase{a}}$  \,\ {\lowercase{and}}\,\ Carlos Orlando Ochoa C.$^{\lowercase{b}}$}
\dedicatory{$^a$  Universidad Sergio Arboleda\\ 
 $^b$  Universidad Distrital Francisco Jose de Caldas} 
\email{$^a$jluna@ima.usergioarboleda.edu.co}
\email{$^b$ oochoac@udistrital.edu.co}
\footnote{Partially supported by Universidad de Cartagena}
\begin{abstract}
U. H{\"o}hle and A. {\v{S}}ostak  have developed in \cite{HS}  the  category  of complete quasi-monoidal lattices; S. E. Rodabaugh in \cite{RO} proposed its opposite category and with   a subcategory $\mathbf{C}$ of the  latter,  he define grounds of the form  $\mathbf{SET\times C}$.

In this paper, for each ground category of the form  $\mathbf{SET\times C}$, we study categorical frameworks for variable-basis fuzzy filters, particularly  the category $\mathbf{C-FFIL}$ of variable-basis fuzzy filters, as a natural generalization of the category of fixed-basis fuzzy filters which was introduced in \cite{LO}. In addition, we get some relations between the category of variable-basis fuzzy filters and the category  of variable-basis fuzzy  topological  spaces.
\end{abstract}
\maketitle 
\baselineskip=1.7\baselineskip
\section*{0. Introduction}
The mathematical theories being created in  fuzzy sets are determined by certain data, incluiding the following: the ground (or base) category for a set theory, the powerset operators of the graound category, the forgetful (or underlying) functor  from  the concrete category   into the ground category, and the particular structure carried by the objets of the concrete category and preserved by the morphisms of that category (cf.\cite{AD} and \cite{RO}). \newline 
Variable-basis fuzzy filter (as variable-basis fuzzy topology in \cite{RO}) request grounds which are products of two categories: The category $\mathbf{SET}$, which carries ``point-set" information, and the second category, which in this paper will be a subcategory $\mathbf{C}$ of the category $\mathbf{LOQML\equiv CQML^{op}}$. The category $\mathbf{CQML}$ of complete quasi-monoidal lattices was being developed by U. H{\"o}hle and A. {\v{S}}ostak in \cite{HS}.\newline
Following P. T. Johnstone (\cite{Jhst}), within the text of the paper, those propositions, lemmas and theorems whose proofs require Zorn's lemma are distinguished by being marked with an asterisk.
\section{Ground categories $\mathbf{SET\times C}$ and $\mathbf{SET\times L_{\Phi}}$}
\subsection{Ground categories}
S. E. Rodabaugh presents in \cite{RO}   the ground categories $\mathbf{SET\times C}$ and $\mathbf{SET\times L_{\Phi}}$ as follows:
\begin{definition}\cite{RO}\ 
Let $\mathbf{C}$ be a subcategory of $\mathbf{LOQML}$. The ground category  $\mathbf{SET\times C}$ comprises the following data:
\begin{enumerate}
\item[(SC1)]{\bf  Objects.} $(X,L)$, where $X\in \mathbf{|SET|}$ and $ L\in \mathbf{|C|}$. The object $(X,L)$ is a {\bf (ground) set}.
\item[(SC2)]{\bf Morphisms.} $\left( f,\Phi \right) : (X,L)\rightarrow (Y,M)$, where $f:X\rightarrow Y$ in $\mathbf{SET}$,  $\Phi:L\rightarrow M $ in $\mathbf{C}$. The morphism $\left( f,\Phi \right)$ is a {\bf(ground) function}.
\item[(SC3)] {\bf Composition.} Component-wise in the respective categories.
\item[(SC4)] {\bf Identities.} Component-wise in the respective categories, i.e. \linebreak $id_{(X,L)} = (id_X, id_L)$.
\end{enumerate}
\end{definition}
\begin{definition}\cite{RO}\ 
Let $\mathbf{C}$  be a  subcategory of $\mathbf{LOQML}$, $L\in \mathbf{|C|}$ and
\linebreak
 $\Phi\in Hom(\mathbf{C},\mathbf{C})$. The ground category  $\mathbf{SET\times L_{\Phi}}$ comprises the following data:
\begin{enumerate}
\item[(SL1)]{\bf  Objects.} $(X,L)$, where $X\in \mathbf{|SET|}$. The object $(X,L)$ es un {\bf (ground) set}.
\item[(SL2)]{\bf Morphisms.} $\left( f,\Phi \right) : (X,L)\rightarrow (Y,L)$, where $f:X\rightarrow Y$ in $\mathbf{SET}$. The morphism  $\left( f,\Phi \right)$ is a {\bf (ground) function}.
\item[(SL3)] {\bf Composition.} $(f,\Phi)\circ (g,\Phi)=(f\circ g, \Phi)$, where the composition in the first component is in $\mathbf{SET}$.
\item[(SL4)] {\bf Identities.} $id_{(X,L)} = (id_X, \Phi)$, where the identity in the first component is in $\mathbf{SET}$.
\end{enumerate}
\end{definition}
\begin{remark}[{\bf Powerset operators}]
For $\left( f,\Phi \right) : (X,L)\rightarrow (Y,M)$ in $\mathbf{SET\times C}$ or   $\mathbf{SET\times L_{\Phi}}$, the forward powerset operator,\linebreak  $\left( f,\Phi \right)^{\rightarrow}:L^X\rightarrow M^Y$ is defined by
\[
\left( f,\Phi \right)^{\rightarrow}(a) = \bigwedge \{b\mid f_L^{\rightarrow}(a)\leqslant \left( \Phi^{op} \right)(b)\},
\]
where $ f_L^{\rightarrow}:L^X\rightarrow L^Y$ is defined by $ f_L^{\rightarrow}(a)(y)=\bigvee_{f(x)=y}a(x)$, and \linebreak $\left( \Phi^{op} \right):L^Y\leftarrow M^Y$ is defined by $\left( \Phi^{op} \right)(b)=\Phi^{op}\circ b$.
The backward powerset operator  $\left( f,\Phi \right)^{\leftarrow} : L^X\leftarrow M^Y$ is defined by
\[
\left( f,\Phi \right)^{\leftarrow}(b) =  \Phi^{op} \circ b\circ f,
\]
in other words, that diagram
\[
\begin{diagram}
\node{X}\arrow{e,t}{f}\arrow{s,l}{\left( f,\Phi \right)^{\leftarrow}(b)}\node{Y}\arrow{s,r}{b}\\
\node{L}\node{M}\arrow{w,b}{\Phi^{op}}
\end{diagram}
\]
is commutative.
\end{remark}
\begin{remark}
\begin{enumerate}\
\item Note that  $\mathbf{SET\times L_{\Phi}}$ is a subcategory of  $\mathbf{SET\times C}$ if and only if  $\Phi = id.$
\item When $\Phi^{op}$ preserves  arbitrary meets, it has a co-adjoint (or left adjoint):
\[
^{*}\Phi:L\rightarrow M,\,\,\,\,\,\ \text{defined by}\,\,\ ^{*}\Phi(\alpha)= \bigwedge\{ \beta \in M\mid \alpha\leqslant \Phi^{op}(\beta)\}
\]
\item Clearly, $\left( f,\Phi \right)^{\leftarrow}$ and $\left( f,\Phi \right)^{\rightarrow}$ constitute an adjoint pair.
\item $\left( f,\Phi \right)$ is an isomorphism in $\mathbf{SET\times C}$ iff $f$ and $\Phi$ are isomorphisms in $\mathbf{SET}$ and $\mathbf{ C}$, respectively, iff $f$ and $\Phi$ are bijections.
\end{enumerate}
\end{remark}

\section{The category $\mathbf{C-FFIL}$}
The categorical frameworks detailed below rest on a ground category of the form $\mathbf{SET\times C}$, where $\mathbf{C}$ is a subcategory of $\mathbf{LOQML}$. This section  generalizes  the category $\mathbf{C-FIL}$ studied in \cite{LO} for this kind of ground categories. 
\begin{definition}\label{c-ffil}
Let $\mathbf{C}$  be a  subcategory of  $\mathbf{LOQML}$. The category  $\mathbf{C-FFIL}$ consists of the following data:
\begin{enumerate}
\item[(FF1)]{\bf  Objects:} Ordered triples $(X,L,\digamma)$ satisfying the following axioms:
\begin{enumerate}
\item  $(X,L) \in |\mathbf{SET\times C}|$.
\item  $\digamma: L^X\rightarrow L$ is a mapping, called a {\bf fuzzy filter}, satisfying:
\begin{enumerate}
\item $\digamma(1_X)=\top$, where $1_X: X\rightarrow L$ is the constant mapping  defined by $1_X(t)=\top,\,\,\ \text{for all}\,\ t\in X$.
\item $\digamma(0_X)=\bot$, where $0_X: X\rightarrow L$ is the constant mapping defined by $1_X(t)=\bot,\,\,\ \text{for all}\,\ t\in X$.
\item  $f\leqslant g$ implies $ \digamma(f)\leqslant \digamma(g)$, for all  $f, g \in L^X$
\item $\digamma(f)\otimes \digamma(g)\leqslant \digamma(f\otimes g)$, for all $f, g \in L^X$
\end{enumerate}
\item {\bf Equality of objects:} $(X,L,\digamma)=(Y, M, \varOmega)$\ if and only if\linebreak $(X,L)=(Y, M)$\ en \ $\mathbf{SET\times C}$,\ and \ $\digamma = \varOmega$\ as $\mathbf{SET}$-mappings from  $L^X\equiv M^Y$ to $L\equiv M$.
\end{enumerate}
\item[(FF2)]{\bf Morphisms:} Ordered pairs  $\left( f,\phi \right) : (X,L,\digamma)\rightarrow (Y, M,\varOmega)$, called {\bf fuzzy filter continuous morphisms}, satisfying the following axioms: 
\begin{enumerate}
\item  $\left( f,\phi \right) : (X,L)\rightarrow (Y,M)$ is a morphism in $\mathbf{SET\times C}$
\item $\phi^{op} \circ \varOmega \leqslant \digamma\circ(f,\phi)^{\leftarrow}$ on $M^Y$.
\end{enumerate}
\item[(FF3)] The composition of morphisms is realized as in the category $\mathbf{SET\times C}$.
\end{enumerate}
$(X, L, \digamma)$ is called a {\bf fuzzy filtered set}.
\end{definition}

An alternate expression  for fuzzy filter continuous morphisms is
\begin{proposition}
On  $M^Y$ the following holds:
\[
\phi^{op} \circ \varOmega\leqslant \digamma\circ(f,\phi)^{\leftarrow}\,\,\ \text{ if and only if}\,\,\,\ \varOmega \leqslant (\phi^{op})^{*}\circ \digamma\circ (f,\phi)^{\leftarrow},
\]
where $(\Phi^{op})^{*}$ is the right (Galois) adjoint of $\Phi^{op}$.
\end{proposition}
\begin{proof}
For necessity, we apply $(\Phi^{op})^{*}$ to both sides, invoking the fact that $\left((\Phi^{op})^{*}, \Phi^{op}\right)$ is a Galois correspondence; and for suficiency, we apply $\Phi^{op}$ to both sides, invoking again  that $\left((\Phi^{op})^{*}, \Phi^{op}\right)$ is a Galois correspondence.
\end{proof}
\begin{proposition}
Let $\mathbf{C}$ be a subcategory of \ $\mathbf{LOQML}$. Then   $\mathbf{C-FFIL}$ is a concrete category over $\mathbf{SET\times C}$.
\end{proposition}
\begin{proof}
Clearly, the forgetful functor $V:\mathbf{C-FFIL}\rightarrow \mathbf{SET\times C}$ is faithful. Then, the main point to be checked is composition. Let the following morphisms be given in  $\mathbf{C-FFIL}$:
\[
\left( f,\phi \right) : (X,L,\digamma_1)\rightarrow (Y,M,\digamma_2)\,\,\ and \,\,\,\ \left( g,\psi \right): (Y,M,\digamma_2)\rightarrow (Z,N ,\digamma_3),
\]
then 
\begin{align*}
(\psi\circ\phi)^{op}\circ \digamma_3&= \phi^{op}\circ\psi^{op}\circ \digamma_3\\
&\leqslant \phi^{op}\circ\digamma_2\circ(g,\psi)^{\leftarrow}\\
&\leqslant \digamma_1\circ (f,\phi)^{\leftarrow}\circ(g,\psi)^{\leftarrow}\\
&=\digamma_1\circ( g\circ f, \psi \circ\phi)^{\leftarrow}.
\end{align*}
This shows that  $( g\circ f, \psi \circ\phi)$ is a fuzzy filter continuous morphisms.
\end{proof}
\begin{proposition}
Let $(X,L)$ be a $\mathbf{SET\times C}$-object\,\ and  let $(X,L,\digamma_{\lambda})_{\lambda \in \Lambda}$ be a non-empty family of  fuzzy filtered sets. Then the mapping  
\[
\mathcal F : L^X \rightarrow L \,\,\,\,\,\ \text{defined by}\,\,\,\,\,\ \mathcal F(g)= \bigwedge_{\lambda \in \Lambda}\digamma_{\lambda}(g)
\]
is a fuzzy filter on the ground set  $(X,L)$.
\end{proposition}
\begin{proof}
In fact,
\begin{enumerate}
\item[i.] $\mathcal F(1_X)=\bigwedge_{\lambda \in \Lambda}\digamma_\lambda(1_X)=\bigwedge_{\lambda \in \Lambda}\top =\top$.
\item [ii.]$\mathcal F(0_X)=\bigwedge_{\lambda \in \Lambda}\digamma_\lambda(0_X)=\bigwedge_{\lambda \in \Lambda}\bot =\bot$.
\item[iii.] If $f\leqslant g$\,\ then \,\ $\digamma_\lambda(f)\leqslant \digamma_\lambda(g)$\,\ for each\,\ $\lambda \in \Lambda$,\,\ therefore
\[
\mathcal F(f)=\bigwedge_{\lambda \in \Lambda}\digamma_\lambda(f)\leqslant \bigwedge_{\lambda \in \Lambda}\digamma_\lambda(g)=F(g).
\]
\item[iv.] Finally, for each $f, g\in L^X$\,\ we have
\begin{equation*}
\begin{split}
\mathcal F(f)\otimes \mathcal F(g)&=\bigwedge_{\lambda \in \Lambda}\digamma_\lambda(f)\otimes \bigwedge_{\lambda \in \Lambda}\digamma_\lambda(g) \leqslant \bigwedge_{\lambda \in \Lambda}\left( \digamma_\lambda(f)\otimes \digamma_\lambda(g)\right)\\
&\leqslant \bigwedge_{\lambda \in \Lambda}\digamma_\lambda(f\otimes g)=\mathcal F(f\otimes g).
\end{split}
\end{equation*}
\end{enumerate}
\end{proof}
\begin{theorem}
Let $\mathbf{C}$ be a subcategory of $\mathbf{LOQML}$, let\linebreak $\left( f,\Phi \right) : (X,L)\rightarrow (Y,M)$  in $\mathbf{SET\times C}$, let $(\Phi^{op})^{*}$ be the right adjoint of $\Phi^{op}: L\leftarrow M$ and let $(X,L,\digamma)$ be a fuzzy filtered set. Then the following holds:
\begin{enumerate}
\item $\digamma_{(f,\Phi)}:= (\Phi^{op})^{*}\circ \digamma\circ(f,\Phi)^{\leftarrow}:M^Y\rightarrow M$ is a fuzzy filter,
\item $\left( f,\Phi \right) : (X,L,\digamma)\rightarrow (Y,M, \digamma_{(f,\Phi)})$ is a fuzzy filter continuous morphisms,
\item $\left( f,\Phi \right) : (X,L,\digamma)\rightarrow (Y,M, \digamma_1)$ is a fuzzy filter continuous morphism if and only if  $\digamma_1\leqslant \digamma_{(f,\Phi)}$,
\item $\left( f,\Phi \right) : (X,L,\digamma)\rightarrow (Y,M, \digamma_{(f,\Phi)})$ is a final morphism in $\mathbf{C-FFIL}$,
\item $\left( f,\Phi \right) : (X,L,\digamma)\rightarrow (Y,M, \digamma^{'})$ is a final morphism in $\mathbf{C-FFIL}$ if and only if   $\digamma^{'}\leqslant \digamma_{(f,\Phi)}$.
\end{enumerate}
\end{theorem}
\begin{proof} 
\begin{enumerate}
\item We check the fuzzy filter  axioms of $\digamma_{(f,\Phi)}$ as follows:
\begin{enumerate}
\item
\begin{align*}
\digamma_{(f,\Phi)}(1_Y)
&=\left [ (\Phi^{op})^{*}\circ \digamma\circ(f,\Phi)^{\leftarrow}\right](1_Y)\\
&=(\Phi^{op})^{*}\left( \digamma(\Phi^{op}\circ 1_Y\circ f)\right)\\
&=(\Phi^{op})^{*}(\digamma(1_X),\ (\text{$\Phi$ is a morphism  in $\mathbf{CQML}$})\\
&=(\Phi^{op})^{*}(\top)= \top,\ (\text{$(\Phi^{op})^{*}$ is the right adjoint of $\Phi^{op}$}).
\end{align*}
\item
\begin{align*}
\digamma_{(f,\Phi)}(O_Y)
&=\left [ (\Phi^{op})^{*}\circ \digamma\circ(f,\Phi)^{\leftarrow}\right](0_Y)\\
&=(\Phi^{op})^{*}\left( \digamma(\Phi^{op}\circ 0_Y\circ f)\right)\\
&=(\Phi^{op})^{*}(\digamma(0_X))=(\Phi^{op})^{*}(\bot)= \bot.
\end{align*}
\item For $h,k \in M^Y$ and $h\leqslant k$, we have
\begin{align*}
\digamma_{(f,\Phi)}(h)
&= \left [ (\Phi^{op})^{*}\circ \digamma\circ(f,\Phi)^{\leftarrow}\right](h)\\
&=(\Phi^{op})^{*}\left( \digamma(\Phi^{op}\circ h\circ f)\right)\\
&\leqslant(\Phi^{op})^{*}\left( \digamma(\Phi^{op}\circ k\circ f)\right)=\digamma_{(f,\Phi)}(k).
\end{align*}
\item Let $h,k \in M^Y$ then we have
\begin{align*}
& \digamma_{(f,\Phi)}(h)\otimes  \digamma_{(f,\Phi)}(k)\\
&= \left [ (\Phi^{op})^{*}\circ \digamma\circ(f,\Phi)^{\leftarrow}\right](h)\otimes\left [ (\Phi^{op})^{*}\circ 
\digamma\circ(f,\Phi)^{\leftarrow}\right](k) \\
&=(\Phi^{op})^{*}\left( \digamma(\Phi^{op}\circ h\circ f)\right)\otimes (\Phi^{op})^{*}\left( \digamma
(\Phi^{op}\circ k\circ f)\right)\\
&=(\Phi^{op})^{*}\left[ \digamma(\Phi^{op}\circ h\circ f)\otimes \digamma(\Phi^{op}\circ k\circ f)\right]\\
&\leqslant(\Phi^{op})^{*}\left[ \digamma\left\{(\Phi^{op}\circ h\circ f)\otimes (\Phi^{op}\circ k\circ f)\right\}\right]\\
&\leqslant(\Phi^{op})^{*}\left[ \digamma\left\{(\Phi^{op}\circ (h\otimes k)\circ f)\right\}\right]\\
&\leqslant(\Phi^{op})^{*}\left[ \digamma\left\{(\Phi^{op}\circ h\circ f)\otimes (\Phi^{op}\circ k\circ f)\right\}\right]\\
&=(\Phi^{op})^{*}\left[ \digamma\left\{\Phi^{op}\circ (h\otimes  k)\circ f\right\}\right]\\
&= \left [ (\Phi^{op})^{*}\circ \digamma\circ(f,\Phi)^{\leftarrow}\right](h\otimes k)\\
&=\digamma_{(f,\Phi)}(h\otimes k)
\end{align*}
\end{enumerate}
\item Since $\Phi^{op}\circ (\Phi^{op})^{*}\leqslant id_L$, we have 
\begin{align*}
(\Phi^{op})\circ  \digamma_{(f,\Phi)}&= \Phi^{op}\circ(\Phi^{op})^{*}\circ \digamma\circ(f,\Phi)^{\leftarrow}\\
&\leqslant \digamma\circ(f,\Phi)^{\leftarrow},
\end{align*}
thus  $\left( f,\Phi \right) : (X,L,\digamma)\rightarrow (Y,M, \digamma_{(f,\Phi)})$ is a fuzzy filter continuous morphisms.
\item Let  $\left( f,\Phi \right) : (X,L,\digamma)\rightarrow (Y,M, \digamma_1)$ be a fuzzy filter continuous morphisms. Then we have 
\[
\Phi^{op}\circ \digamma_1\leqslant \digamma\circ (f,\Phi)^{\leftarrow},
\]
which implies that
\begin{align*}
\digamma_1&\leqslant (\Phi^{op})^{*}\circ \Phi^{op}\circ \digamma_1\\
& \leqslant (\Phi^{op})^{*}\circ\digamma\circ(f,\Phi)^{\leftarrow} \\
&= \digamma_{(f,\Phi)}.
\end{align*}
For necessity, let $ \digamma_1\leqslant\digamma_{(f,\Phi)}$. Since  $\left( f,\Phi \right) : (X,L,\digamma)\rightarrow (Y,M, \digamma_{(f,\Phi)})$ is a fuzzy filter continuous morphisms, we have
\begin{align*}
\Phi^{op}\circ \digamma_1&\leqslant \Phi^{op}\circ  \digamma_{(f,\Phi)}\\
& \leqslant \digamma\circ(f,\Phi)^{\leftarrow}. \\
\end{align*}
\item To show  $\left( f,\Phi \right) : (X,L,\digamma)\rightarrow (Y,M, \digamma_{(f,\Phi)})$ is a final morphism in  $\mathbf{C-FFIL}$, we must verify that for each  $(Z,N, \digamma^{'})\in \mathbf{C-FFIL}$, and for each $(g, \Psi)\in Hom \left( (Y,M), (Z, N)\right)$ in $ \mathbf{SET\times C}$, the following holds:
\[
(g,\Psi)\circ (f,\Phi) \in Hom_{\mathbf{C-FFIL}} \left( (X,L,\digamma), (Z,N,\digamma^{'})\right)
\]
implies
\[
(g,\Psi) \in Hom_{\mathbf{C-FFIL}} \left( (Y, M,\digamma_{(f,\Phi)}), (Z,N,\digamma^{'})\right).
\]
Since $(g,\Psi)\circ (f,\Phi)$ is a fuzzy filter continuous morphisms, it follows
\[
\Phi^{op}\circ\Psi^{op}\circ \digamma^{'}\leqslant \digamma\circ (f,\Phi)^{\leftarrow}\circ(g,\Psi)^{\leftarrow}.
\]
Now from the definition of  $\digamma_{(f,\Phi)}$ it  follows
\begin{align*}
\Psi^{op}\circ \digamma^{'}&\leqslant (\Phi^{op})^{*}\circ \Phi^{op}\circ \Psi^{op}\circ \digamma^{'}\\
& \leqslant (\Phi^{op})^{*}\circ\digamma\circ (f,\Phi)^{\leftarrow}\circ(g,\Psi)^{\leftarrow}\\
&= \digamma_{(f,\Phi)}\circ (g,\Psi)^{\leftarrow}.
\end{align*}
\item Sufficiency is as in the previous case; for necessity, let\linebreak $\left( f,\Phi \right) : (X,L,\digamma)\rightarrow (Y,M, \digamma^{'})$ be a final morphism in $\mathbf{C-FFIL}$. Since  $\left( f,\Phi \right)$ is assumed to be a fuzzy filter continuous morphisms; and so by (3) we have that $\digamma^{'}\leqslant \digamma_{(f,\Phi)}$. Using  (5) and (2), and  as a consequence of the  finality we can conclude that\linebreak  $\left( id_Y,id_M \right): (Y,M, \digamma^{'})\rightarrow (Y, M,\digamma_{(f,\Phi)})$ is a a fuzzy filter continuous morphisms, which implies that
\begin{align*}
\digamma_{(f,\Phi)}&= id^{op}\circ \digamma_{(f,\Phi)} \\
& \leqslant \digamma^{'}\circ(id_Y,id_M)^{\leftarrow} \\
&=\digamma^{'}.
\end{align*}
Hence $\digamma^{'}=\digamma_{(f,\Phi)}. $
\end{enumerate}
This completes the proof of the proposition.
\end{proof}
\begin{theorem}[Initial fuzzy filter]
Let\ $\mathbf{C}$ be a subcategory of\,\ $\mathbf{LOQML}$, let  $\left( f,\Phi \right) : (X,L)\rightarrow (Y,M)$ be a   morphism in $\mathbf{SET\times C}$, where $f:X\rightarrow Y$ is an onto map, \ and let \ $\digamma^{'}: M^Y\rightarrow M $ be a fuzzy filter on $(Y,M)$, then
\[
\digamma \equiv \Phi^{op}\circ \digamma^{'} \circ \left( f,\Phi \right)^{\rightarrow}: L^X\rightarrow L
\]
is a fuzzy filter on $(X,L)$.
\end{theorem}
\begin{proof} In fact, we have
\begin{itemize}
\item[i).] 
\begin{align*}
\digamma(1_X)&= \Phi^{op}\left[ \digamma^{'}\left( ^{*}\Phi \circ f_{L^{\rightarrow}}(1_{(X,L)}) \right)\right]\\
&=\Phi^{op}\left[ \digamma^{'}\left( ^{*}\Phi (1_{(Y,L)}) \right)\right]\\
&=\Phi^{op}\left[ \digamma^{'}(1_{(Y,M)})\right]=\Phi^{op}(\top)=\top.
\end{align*}
\item[ii).] 
\begin{align*}
\digamma(1_X)&= \Phi^{op}\left[ \digamma^{'}\left( ^{*}\Phi \circ f_L^{\rightarrow}(0_{(X,L)}) \right)\right]\\
&=\Phi^{op}\left[ \digamma^{'}\left( ^{*}\Phi (0_{(Y,L)}) \right)\right]\\
&=\Phi^{op}\left[ \digamma^{'}(0_{(Y,M)})\right]=\Phi^{op}(\bot)=\bot.
\end{align*}
\item[iii).] Let $h_1\leqslant h_2$, the goal is to show that $\digamma(h_1)\leqslant\digamma(h_2)$. Since for $y\in Y$ we have
\begin{align*}
(^{*}\Phi \circ\left( f_L^{\rightarrow}(h_1))\right)(y)&=  ^{*}\Phi (\bigvee_{f(x)=y} h_1(x))\\
&=\bigwedge\left\{m\in M \mid\bigvee_{f(x)=y} h_1(x)\leqslant\Phi^{op}(m)\right\}\\
&\leqslant\bigwedge\left\{m\in M \mid\bigvee_{f(x)=y} h_2(x)\leqslant\Phi^{op}(m)\right\}\\
&=  ^{*}\Phi (\bigvee_{f(x)=y} h_2(x))=(^{*}\Phi \left(\circ f_L^{\rightarrow}(h_2))\right)(y)
\end{align*}
therefore $^{*}\Phi \circ \left(f_L^{\rightarrow}(h_1)\right)\leqslant ^{*}\Phi \circ \left(f_L^{\rightarrow}(h_2)\right)$. Hence, since $\digamma^{'}$ and  $ \Phi^{op}$ are isotone, we have
\begin{align*}
\digamma(h_1)&= \Phi^{op}\left[ \digamma^{'}\left( ^{*}\Phi \circ\left( f_L^{\rightarrow}(h_1)\right) \right)\right]\\
&\leqslant \Phi^{op}\left[ \digamma^{'}\left( ^{*}\Phi \circ \left(f_L^{\rightarrow}(h_2)\right) \right)\right]\\
&=\digamma(h_2).
\end{align*}
\item[iv).] Let $h$ and $k$ be two elements of $L^X$, in order to show that \linebreak $\digamma(h)\otimes \digamma(k)\leqslant \digamma(h\otimes k)$ we proceed as follows:\newline
Firstly, for each $y\in Y$,
\begin{align*}
\left[f_L^{\rightarrow}(h\otimes k)\right](y)&=\bigvee_{f(x)=y}(h\otimes k)(x)\\
&=\bigvee_{f(x)=y}h(x)\otimes k(x)\\
&=\bigvee_{f(x)=y}h(x)\otimes\bigvee_{f(x)=y} k(x)\\
&=(f_L^{\rightarrow}(h))(y)\otimes (f_L^{\rightarrow}(h))(y); 
\end{align*}
consequently, $f_L^{\rightarrow}(h\otimes k)=f_L^{\rightarrow}(h)\otimes f_L^{\rightarrow}(k)$; and so
\[
^{*}\Phi \circ\left( f_L^{\rightarrow}(h\otimes k) \right)= \left( ^{*}\Phi \circ f_L^{\rightarrow}(h) \right)\otimes \left( ^{*}\Phi \circ f_L^{\rightarrow}(k) \right).
\]
Finally, since $\digamma^{'}$ is a fuzzy filter,
\begin{align*}
&\digamma^{'}\left[\left( ^{*}\Phi \circ f_L^{\rightarrow}(h) \right)\right]\otimes \digamma^{'}\left[\left( ^{*}\Phi \circ f_L^{\rightarrow}(k) \right)\right]\\
&\leqslant \digamma^{'}\left[\left( ^{*}\Phi \circ f_L^{\rightarrow}(h) \right)\otimes \left( ^{*}\Phi \circ f_L^{\rightarrow}(k) \right)\right]\\
&=\digamma^{'}\left[ ^{*}\Phi \circ\left( f_L^{\rightarrow}(h\otimes k) \right)\right].\\
\end{align*}
Therefore
\begin{align*}
\digamma(h)\otimes \digamma(k)&=\Phi^{op}\left[ \digamma^{'}\left( ^{*}\Phi \circ f_L^{\rightarrow}(h) \right)\right]\otimes 
\Phi^{op}\left[ \digamma^{'}\left( ^{*}\Phi \circ f_L^{\rightarrow}(k) \right)\right]\\
&= \Phi^{op}\left[ \digamma^{'}\left( ^{*}\Phi \circ f_L^{\rightarrow}(h) \right)\otimes 
\digamma^{'}\left( ^{*}\Phi \circ f_L^{\rightarrow}(k) \right)\right]\\
&\leqslant \Phi^{op}\left[\digamma^{'}\left[ ^{*}\Phi \circ\left( f_L^{\rightarrow}(h\otimes k) \right)\right] \right]\\
&=\digamma(h\otimes k)
\end{align*}
\end{itemize}
\end{proof}

\subsection{Fuzzy ultrafilters}

Let $\mathfrak F_{FF}(X,L)$ be the set of all fuzzy filters on $(X,L)$. On  $\mathfrak F_{FF}(X,L)$ we introduce a partial ordering \ $\curlyeqprec$\ by
\[
\mathcal F_1\curlyeqprec \mathcal F_2 \Leftrightarrow \mathcal F_1(f)\leqslant \mathcal F_2(f),\,\,\,\,\ \forall (f)\in L^X
\]
\begin{$^{*}$proposition}
The partially ordered set $(\mathfrak F_{FF}(X,L),\curlyeqprec)$ has maximal elements.
\end{$^{*}$proposition}
\begin{proof}
Referring to Zorn's lemma, it is sufficient to show  that every chain $\mathcal C$ in $\mathfrak F_{FF}(X,L)$ has an upper bound in $\mathfrak F_{FF}(X,L)$. For this purpose let us consider a non-empty chain $\mathcal C=\{\mathcal F_{\lambda}\mid \lambda \in I\} $. We define a map\linebreak $\mathcal F_{\infty}: L^X\rightarrow L$ by
\[
\mathcal F_{\infty}(f)=\bigvee_{\lambda\in I}\mathcal F_{\lambda}(f),
\]
and we show that $\mathcal F_{\infty}$ is a fuzzy filter on $(X,L)$. In fact
\begin{enumerate}
\item[(FF1.b.i)] $\mathcal F_{\infty}(1_X)=\bigvee_{\lambda\in I}\mathcal F_{\lambda}(1_X)=\bigvee_{\lambda\in I}\top=\top$.
\item[(FF1.b.ii)]  $\mathcal F_{\infty}(0_X)=\bigvee_{\lambda\in I}\mathcal F_{\lambda}(0_X)=\bigvee_{\lambda\in I}\bot=\bot$.
\item[(FF1.b.iii)] $f\leqslant g\Rightarrow  \mathcal F_{\infty}(f)=\bigvee_{\lambda\in I}\mathcal F_{\lambda}(f)\leqslant\bigvee_{\lambda\in I}\mathcal F_{\lambda}(g) =\mathcal F_{\infty}(g)$.
\item[(FF1.b.iv)]
\begin{align*}
\mathcal F_{\infty}(f)\otimes\mathcal F_{\infty}(g)&= \left(\bigvee_{\lambda\in I}\mathcal F_{\lambda}(f)\right)\otimes \left(\bigvee_{\lambda\in I}\mathcal F_{\lambda}(g)\right)\\
&=\bigvee_{\lambda\in I}\left[\mathcal F_{\lambda}(f)\otimes\mathcal F_{\lambda}(g)\right]\\
&\leqslant \bigvee_{\lambda\in I}\left[\mathcal F_{\lambda}(f\otimes g)\right]\\
&=\mathcal F_{\infty}(f\otimes g).
\end{align*}
\end{enumerate}
\end{proof}
\begin{definition}
A maximal element in $(\mathfrak F_{FF}(X),\curlyeqprec)$ is also called a fuzzy ultrafilter.
\end{definition}

\begin{$^{*}$proposition}\label{UF}
For every  fuzzy filter $\mathcal U:  L^X\rightarrow L$  on $X$ the following assertions are equivalent
\begin{enumerate}
\item[(i)] $\mathcal U$ is a fuzzy ultrafilter.
\item[(ii)] $\mathcal U(f)=\left[\mathcal U(f\boldsymbol{\impl}0_X)\right]\impl\bot,\,\ \text{for all}\,\ f\in L^X. $
\end{enumerate}
\end{$^{*}$proposition}
\begin{proof} $(i)\Rightarrow(ii)$\newline
Because of $(FF1.b.iii)$ and $(FF1.b.iv)$  every fuzzy filter satisfies the condition
\begin{enumerate}
\item[(FF1.b.ii')]  $\mathcal U(f)\leqslant\left[\mathcal U\left(f\boldsymbol{\impl} 0_X\right)\right]\impl\bot,\,\ \text{for all}\,\ f\in L^X.$
\end{enumerate}
In order to verify $(i)\Rightarrow (ii)$ it is sufficient to show that the maximality of $\mathcal U$ implies
\[
\left[\mathcal U\left(f\boldsymbol{\impl} 0_X\right)\right]\impl\bot\leqslant \mathcal U(f),\,\ \forall f\in L^X.
\]
For this purpose, we fix an element $g\in L^X$, for that element we let\linebreak  $\mathcal G_{g}:= \left[\mathcal U\left(g\boldsymbol{\impl} 0_X\right)\right]\impl\bot$ and define a map\ $\hat{\mathcal U}:  L^X\rightarrow L$ \ by
\[
\hat{\mathcal U}(f)=\mathcal U(f)\bigvee\Bigl\{\mathcal U\left(g\boldsymbol{\impl} f\right)\otimes \mathcal G_{g}\Bigr\}.
\]
We must show that $\hat{\mathcal U}$ is a fuzzy ultrafilter.
Firstly $\hat{\mathcal U}$ is a fuzzy filter: obviously
$\hat{\mathcal U}$ satisfies $(FF´1.b.i)$. \\

In order to verify $(FF1.b.ii)$, we have that
\begin{align*}
\hat{\mathcal U}(0_X)&=\mathcal U(0_X)\bigvee\Bigl\{\mathcal U\left(g\boldsymbol{\impl} 0_X\right)\otimes \mathcal G_{g}\Bigr\}\\
&=\bot\lor\Bigl\{\mathcal U\left(g\boldsymbol{\impl} 0_X\right)\otimes \mathcal G_{g}\Bigr\}\\
&=\mathcal U\left(g\boldsymbol{\impl} 0_X\right)\otimes \mathcal G_{g}.\\
\end{align*}
Now we invoke the residuation property of $(L,\leqslant,\otimes)$ to obtain
\[
\hat{\mathcal U}(0_X)=\Bigl\{\mathcal U\left(g\boldsymbol{\impl} 0_X,\right)\otimes \mathcal G_{g}\Bigr\}=\bot.
\]

For the axiom $(FF1.b.iii)$, from
the definition
\[
\hat{\mathcal U}(f)=\mathcal U(f)\bigvee\Bigl\{\mathcal U\left(g\boldsymbol{\impl} f\right)\otimes \mathcal G_{g}\Bigr\}
\]
and
\[
\hat{\mathcal U}(h)=\mathcal U(h)\bigvee\Bigl\{\mathcal U\left(g\boldsymbol{\impl} h\right)\otimes \mathcal G_{g}\Bigr\}.
\]
Now, for $f\leqslant h$ we have that $\mathcal U(f)\leqslant\mathcal U(h)$,
moreover,
\begin{align*}
g\boldsymbol{\impl} f &=\bigvee\bigl\{k\in L^X\mid g\otimes k\curlyeqprec f\bigr\}\\
&\leqslant \bigvee\bigl\{k\in L^X\mid g\otimes k\curlyeqprec h\bigr\}\\
&=g\boldsymbol{\impl} h,
\end{align*}
which implies that
\begin{align*}
\hat{\mathcal U}(f)&=\mathcal U(f)\bigvee\Bigl\{\mathcal U\left(g\boldsymbol{\impl} f\right)\otimes \mathcal G_{g}\Bigr\}\\
&\leqslant\mathcal U(h)\bigvee\Bigl\{\mathcal U\left(g\boldsymbol{\impl} h\right)\otimes \mathcal G_{g}\Bigr\}=\hat{\mathcal U}(h).
\end{align*}
For the axiom $(FF1.b.iv)$, we must verify that
\[
\hat{\mathcal U}(f)\otimes \hat{\mathcal U}(h)\leqslant \hat{\mathcal U}(f\otimes h),
\]
In fact,
\noindent
\begin{align*}
&\hat{\mathcal U}(f)\otimes \hat{\mathcal U}(h)\\
&=\bigl(\mathcal U(f)\bigvee\bigl\{\mathcal U\left(g\boldsymbol{\impl} f\right)\otimes \mathcal G_{g}\bigr\}\bigr)
\otimes \bigl(\mathcal U(h)\bigvee\bigl\{\mathcal U\left(g\boldsymbol{\impl} h\right)\otimes \mathcal G_{g}\bigr\}\bigr)\\
&=\mathcal U(f)\otimes\left[\mathcal U(h)\bigvee\bigl\{\mathcal U\left(g\boldsymbol{\impl} h\right)\otimes \mathcal G_{g}\bigr\}\right]\\
&\bigvee\left[\bigl\{\mathcal U\left(g\boldsymbol{\impl} f\right)\otimes \mathcal G_{g}\bigr\}
\otimes\bigl\{\mathcal U(h)\bigvee\bigl\{\mathcal U\left(g\boldsymbol{\impl} h\right)\otimes \mathcal G_{g}\bigr\}\bigr\}\right]\\
&=\mathcal U(f)\otimes \mathcal U(h)\bigvee\left[\mathcal U(f)\otimes\bigl\{\mathcal U\left(g\boldsymbol{\impl} h\right)\otimes \mathcal G_{g}\bigr\}\right]\\
&\bigvee\left(\bigl\{\mathcal U\left(g\boldsymbol{\impl} f\right)\otimes \mathcal G_{g} \bigr\}\otimes \mathcal U(h)\right)
\bigvee \left(\bigl\{ \mathcal U\left(g\boldsymbol{\impl} f\right)\otimes \mathcal G_{g} \bigr\}
\otimes \bigl\{\mathcal U\left(g\boldsymbol{\impl} h\right)\otimes \mathcal G_{g}\bigr\}\right)\\
&=\mathcal U(f)\otimes \mathcal U(h)\bigvee\left[\bigl\{\mathcal U(f)\otimes\mathcal U\left(g\boldsymbol{\impl} h\right)\bigr\}\otimes \mathcal G_{g}\right]\\
&\bigvee\left(\bigl\{\mathcal U(h)\otimes\mathcal U\left(g\boldsymbol{\impl} f\right)\bigr\} \otimes \mathcal G_{g} \right)
\bigvee \left(\bigl\{ \mathcal U\left(g\boldsymbol{\impl} f\right) \otimes \mathcal U\left(g\boldsymbol{\impl} h\right)\bigr\}\otimes \mathcal G_{g}\right)\\
&\leqslant \mathcal U(f\otimes h)\bigvee\left[\mathcal U\left(f\otimes\left[g\boldsymbol{\impl} h\right] \right)\otimes\mathcal G_{g}\right]\\
&\bigvee\left(\mathcal U\left(h\otimes\left[g\boldsymbol{\impl} f\right]\right)\otimes\mathcal G_{g}\right)
\bigvee\left(\mathcal U\left(\left(g,\boldsymbol{\impl} f\right)\otimes \left[g\boldsymbol{\impl} h\right]\right)\otimes\mathcal G_{g}\right)\\
&\leqslant\mathcal U(f\otimes h)\bigvee\left[\mathcal U\left(g\boldsymbol{\impl}f\otimes h\right)\otimes\mathcal G_{g}\right]\\
&\bigvee\left[\mathcal U\left(g\boldsymbol{\impl}f\otimes h\right)\otimes\mathcal G_{g}\right]
\bigvee\left[\mathcal U\left(g\boldsymbol{\impl}f\otimes h\right)\otimes\mathcal G_{g}\right]\\
&=\mathcal U(f\otimes h)\bigvee\left[\mathcal U\left(g\boldsymbol{\impl}f\otimes h\right)\otimes\mathcal G_{g}\right]\\
&=\hat{\mathcal U}(f\otimes h).
\end{align*}
\vspace{0.5cm}
Now we must show that $\hat{\mathcal U}$ is a fuzzy ultrafilter on $X$. In fact, since
\[
\hat{\mathcal U}(f)=\mathcal U(f)\bigvee\Bigl\{\mathcal U\left(g\boldsymbol{\impl} f\right)\otimes \mathcal G_{g}\Bigr\},
\]
clearly $\mathcal U(f)\leqslant\hat{\mathcal U}(f),\,\,\,\ \forall f\in L^X$, but $\mathcal U$ is a fuzzy ultrafilter on $X$, therefore $\hat{\mathcal U}=\mathcal U$.
In this way
\begin{align*}
\mathcal U(g)&=\mathcal U(g)\bigvee\Bigl\{\mathcal U\left(g\boldsymbol{\impl} g\right)\otimes \mathcal G_{g}\Bigr\}\\
&=\mathcal U(g)\bigvee\Bigl\{\mathcal U(1_X)\otimes \mathcal G_{g}\Bigr\}\\
&=\mathcal U(g)\bigvee\Bigl\{\top\otimes \mathcal G_{g}\Bigr\}\\
&=\mathcal U(g)\lor \mathcal G_{g}.\\
\end{align*}
Therefore,
\[
\mathcal G_{g}=\left[\mathcal U\left(g\boldsymbol{\impl} 0_X\right)\right]\impl\bot\leqslant \mathcal U(g),\,\,\,\ \forall g\in L^X.
\]
From the last inequality and $(FF1.b.ii')$ we obtain $(ii)$.\newline
$(ii)\Rightarrow(i)$\newline
We must verify that if
\[
\mathcal U(f)=\left(\mathcal U\left(f\boldsymbol{\impl}0_X\right)\right)\impl\bot,\,\ \text{for all}\,\ f\in L^X,
\]
then $\mathcal U$ is a fuzzy ultrafilter on $X$.\newline
Suppose  $\mathcal U\leqslant\hat{\mathcal U}$, then
\[
\left(\left[\hat{\mathcal U}\left(f\boldsymbol{\impl}0_X\right)\right]\impl\bot\right)\leqslant \left(\left[\mathcal U\left(f\boldsymbol{\impl}0_X\right)\right]\impl\bot\right),
\]
therefore $\hat{\mathcal U}\leqslant \mathcal U$, consequently $\mathcal U$ is an $L$-fuzzy ultrafilter on $X$.
\end{proof}
\begin{$^{*}$proposition}\label{di} 
Let $\phi: X\rightarrow Y$ be a map and let $\mathcal F:  L^X\rightarrow L$ be  a fuzzy filter on $X$. Then
the map $\phi_{\mathcal U}^{\rightarrow}: L^Y\rightarrow L$ is a fuzzy ultrafilter on $Y$, whenever $\mathcal U$ will be a fuzzy ultrafilter on $X$
\end{$^{*}$proposition}
\begin{proof}
 Let $\mathcal U:  L^X\rightarrow L$ be  a fuzzy ultrafilter on $X$ and let $g\in L^Y$, then
\begin{align*}
\phi_{\mathcal U}^{\rightarrow}(g)&= \mathcal U(g\circ\phi)\\
&= \mathcal U\left((g\circ\phi)\boldsymbol{\impl} 0_X\right)\impl \bot\\
&= \mathcal U[(g\circ\phi)\impl (0_Y\circ \phi)]\impl \bot\\
&= \mathcal U[(g\impl 0_Y)\circ \phi]\impl \bot\\
&= \phi_{\mathcal U}^{\rightarrow}[g\impl 0_Y]\impl \bot.
\end{align*}
We conclude from proposition \ref{UF} that $\phi_{\mathcal U}^{\rightarrow}: L^Y\times L\rightarrow L$ is an $L$-fuzzy ultrafilter on $Y$.
\end{proof}

\section{The category $\mathbf{C-FTOP}$} 
In this section we transcribe some facts about categorical topology, taken from  \cite{RO}, in order to establish (in the next section) a relationship between this category and the category  $\mathbf{C-FFIL}$.  
\begin{definition}[ The category $\mathbf{C-FTOP}$]\label{c-ftop} 
Let $\mathbf{C}$ be a  subcategory of\linebreak $\mathbf{LOQML}$. The category $\mathbf{C-FTOP}$ comprises the following data:
\begin{enumerate}
\item[(CF1)]{\bf  Objects.} Objects are ordered triples $(X,L,\Upsilon)$ satisfying the following  axioms:
\begin{enumerate}
\item {\bf Ground axiom.} $(X,L) \in |\mathbf{SET\times C}|$, 
\item {\bf Fuzzy topological  axiom.} $\Upsilon: L^X\rightarrow L$ is a mapping satisfaying:
\begin{enumerate}
\item For all set of index $J$, para todo $\{ f_{\lambda}\mid \lambda \in J\}\subseteq L^X$, 
\[
\bigwedge_{\lambda \in J} \Upsilon(f_{\lambda})\leqslant \Upsilon\left(\bigvee_{\lambda \in J}f_{\lambda}\right)
\]
\item $\Upsilon(f)\otimes \Upsilon(g)\leqslant \Upsilon(f\otimes g)$,\,\,\ $\forall \ f, g \in L^X$
\item $\Upsilon(1_X)=\top$.
\end{enumerate}
\item {\bf  Equality of objects.} $(X,L,\Upsilon)=(Y, M, \Gamma)$ iff $(X,L)=(Y, M)$ in $\mathbf{SET\times C}$ and $\Upsilon = \Gamma$ as $SET$ mappings from $L^X\equiv M^Y$ to $L\equiv M$.
\end{enumerate}
\item[(CF2)]{\bf Morphisms.} Morphisms are  ordered pairs  
\[
\left( f,\Phi \right) : (X,L,\Upsilon)\rightarrow (Y,M,\Gamma)
\]
called {\bf fuzzy continuous morphisms}, satisfying the following axioms: 
\begin{enumerate}
\item {\bf Ground axiom.} $\left( f,\Phi \right) : (X,L)\rightarrow (Y,M)$ is a morphism in $\mathbf{SET\times C}$
\item {\bf fuzzy continuity axiom} $\Phi^{op} \circ \Gamma \leqslant \Upsilon\circ(f,\Phi)^{\leftarrow}$ on $M^Y$.
\item{\bf Equality of morphisms.} As in $\mathbf{SET\times C}$.
\end{enumerate}
\item[(CF3)] {\bf Composition.} As in $\mathbf{SET\times C}$.
\item[(CF4)] {\bf Identities.}  As in $\mathbf{SET\times C}$.
\end{enumerate}
The ordered triple $(X,L,\Upsilon)$ is a {\bf fuzzy topological space} on the ground set $(X,L)$ .
\end{definition}
\begin{proposition} [Alternate fuzzy continuity axiom] \cite{RO}\ 
 On $M^Y$ the following holds:
\[
\Phi^{op} \circ \Gamma \leqslant \Upsilon\circ(f,\Phi)^{\leftarrow}\,\,\ \text{ if and only if}\,\,\,\ \Gamma \leqslant (\Phi^{op})^{*}\circ \Upsilon\circ (f,\Phi)^{\leftarrow},
\]
where $(\Phi^{op})^{*}$ is the right adjoint of $\Phi^{op}$.
\end{proposition}
\begin{theorem} [Final structures and morphisms for fuzzy topology] \cite{RO}\ 
Let $\mathbf{C}$ be a subcategory of $\mathbf{LOQML}$, let $\left( f,\Phi \right) : (X,L)\rightarrow (Y,M)$  in $\mathbf{SET\times C}$, let $\Phi^{*}\equiv (\Phi^{op})^{*}$ be the right  adjoint of $\Phi^{op}: L\leftarrow M$ and let $\Upsilon$ be a fuzzy topology on $(X,L)$. Then the following holds:
\begin{enumerate}
\item $\Upsilon_{(f,\Phi)}\equiv \Phi^{*}\circ \Upsilon\circ(f,\Phi)^{\leftarrow}: Y\rightarrow M$ is a fuzzy topology on $(Y,M)$;
\item $\left( f,\Phi \right) : (X,L,\Upsilon)\rightarrow (Y,M, \Upsilon_{(f,\Phi)})$ is fuzzy continuous;
\item $\left( f,\Phi \right) : (X,L,\Upsilon)\rightarrow (Y,M, \Upsilon^{'})$ is fuzzy continuous iff $\Upsilon^{'}\leqslant \Upsilon_{(f,\Phi)}$;
\item $\Upsilon_{(f,\Phi)}$ is the join of all the fuzzy topologies $\Upsilon^{'}$ on $M^Y$ for which \linebreak $\left( f,\Phi \right) : (X,L,\Upsilon)\rightarrow (Y,M, \Upsilon^{'})$ is fuzzy continuous;
\item $\left( f,\Phi \right) : (X,L,\Upsilon)\rightarrow (Y,M, \Upsilon_{(f,\Phi)})$ is a final morphism in $\mathbf{C-FTOP}$;
\item $\left( f,\Phi \right) : (X,L,\Upsilon)\rightarrow (Y,M, \Upsilon^{'})$ is a final morphism in  $\mathbf{C-FTOP}$
iff  $\Upsilon^{'}\leqslant \Upsilon_{(f,\Phi)}$.
\end{enumerate}
\end{theorem}
\begin{theorem} [$\mathbf{C-FTOP}$ is a topological category] \cite{RO}\ 
For each subcategory  $\mathbf{C}$   of $\mathbf{LOQML}$, the category  $\mathbf{C-FTOP}$ is topological over $\mathbf{SET\times C}$ with respect to the forgetful functor $V$.
\end{theorem}

\section{From  $\mathbf{C-FFIL}$ to $\mathbf{C-FTOP}$}
There exists a natural relationship between the categories $\mathbf{C-FFIL}$ and  $\mathbf{C-FTOP}$. Our purpose in this section is to describe it, as a generalization of  \cite{LO}.\newline
If we compare the axiom $(CF1.b)$ from definition \ref{c-ftop} of fuzzy topology with axiom  $(FF1.b)$ from definition  \ref{c-ffil} of fuzzy filter, we can see that the condition $(ii)$ of the latter  implies condition $(i)$  of the former, in fact:\newline
Let $J\ne \emptyset$ be an index set and let   $\{ f_{\lambda}\mid \lambda \in J\}\subseteq L^X$, then  we have 
\[ f_{\lambda}\leqslant \bigvee_{\lambda \in J}f_{\lambda}\,\ \text{ for each $\lambda \in J$;}
\]
invoking $(CF1.b.iii)$, we get
\[\digamma( f_{\lambda})\leqslant \digamma (\bigvee_{\lambda \in J}f_{\lambda})\,\ \text{ for each $\lambda \in J$,}
\]
therefore
\[ 
\bigwedge_{\lambda \in J} \digamma (f_{\lambda})\leqslant \digamma\left(\bigvee_{\lambda \in J}f_{\lambda}\right).
\]
Moreover, if we change $(FF1.b.ii)$\ from the definition of fuzzy filter  by $\digamma (0_X)=\top $, it is obtained. 
\begin{proposition}[Fuzzy filtered-type topolgy] 
Let $\digamma: L^X\rightarrow L $ be a fuzzy filter on $(X,L)$. Then the mapping  
$\Upsilon_{\digamma}:L^X\rightarrow L $ defined , for each $g\in L^X$,  by
\[
\Upsilon_{\digamma}(g)=
\begin{cases}
\digamma(g)      &\text{if $g\ne 0_X$}\\
\top             &\text{if $g= 0_X$} 
\end{cases}
\]
is a fuzzy topology on $(X,L)$.
\end{proposition}

\begin{$^{*}$corollary}[Fuzzy ultra-filtered-type topolgy] 
Let $\mathcal U: L^X\rightarrow L $ be a fuzzy ultrafilter on $(X,L)$. Then the mapping  
$\Upsilon_{\mathcal U}:L^X\rightarrow L $ defined , for each $g\in L^X$,  by
\[
\Upsilon_{\mathcal U}(g)=
\begin{cases}
\mathcal U(g)      &\text{if $g\ne 0_X$}\\
\top             &\text{if $g= 0_X$} 
\end{cases}
\]
is a fuzzy ultra-topology on $(X,L)$.
\end{$^{*}$corollary}

Thus, if we have a fuzzy filtered set $(X,L,F)$, we obtain the fuzzy  topological space $(X,L, \mathcal T_F)$;\,\ moreover, if\,\ $\phi$\,\ is a morphism between $(X,L,F)$ and $(Y,M,G)$  in $\mathbf{C-FFIL}$,\,\ the {\it same mappping} $\phi$   is also a morphism in\,\ $\mathbf{C-FTOP}$,\,\  and the diagram 
\[
\begin{diagram}
\node{(X,F)}\arrow{e,t}{}\arrow{s,l}{\phi}\node{(X,\mathcal T_F)}\arrow{s,r}{\phi_{*}=\phi}\\
\node{(Y,G)}\arrow{e,b}{}\node{(Y,\mathcal T_G)}
\end{diagram}
\]
is commutative. Also, we observe that if $\phi \ne \psi$\,\ are morphisms in\,\ $\mathbf{C-FFIL}$\,\ then\,\ $\phi_{*}\ne \psi_{*}$.\,\ In other words,
\begin{theorem}
The funtion\,\ $\mathcal T:\mathbf{C-FFIL} \to  \mathbf{C-FTOP}$\,\ that assigns to each object $(X,L,F)$ of $\mathbf{C-FFIL} $ the object $(X,L,\mathcal T_F)$ of\ $\mathbf{C-FTOP}$,\,\ and to each morphism\,\ $\phi$\,\ in\,\ $\mathbf{C-FFIL} $\,\ the morphism\,\ $\phi_{*}=\phi$\,\ in\,\ $\mathbf{C-FTOP}$\,\ is a faithful functor between the category \,\ $\mathbf{C-FFIL} $\,\ of fuzzy filtered sets,   and the category  \,\ $\mathbf{C-FTOP}$\,\ of fuzzy topological spaces.
\end{theorem}

\vspace{2cm}
\end{document}